\documentclass[leqno,12pt]{article}
\usepackage{fullpage}
\usepackage{amssymb}
\usepackage{amsmath}
\usepackage{amsthm}
\newtheorem{theorem}{Theorem}
\newtheorem{problem}{Problem}

\newtheorem{lemma}{Lemma}
\begin{document}

\title{\large\bf On the elementary symmetric functions of $1, 1/2, \ldots , 1/n$\footnote{This work was supported by the National
Natural Science Foundation of China, Grant Nos. 11071121 and
10901002.} \footnote{ Email: ygchen@njnu.edu.cn; tmzzz2000@163.com
}}
\date{}
\author{Yong-Gao Chen\\ \small School of Mathematical Sciences and Institute of Mathematics,
\\ \small Nanjing
Normal
University, Nanjing 210046, P. R. CHINA\\
  Min Tang\\ \small Department of Mathematics, Anhui Normal University,\\
\small Wuhu 241000, P. R. CHINA}
 \maketitle

\begin{abstract} In 1946, P. Erd\H os and I. Niven proved that
there are only finitely many positive integers $n$ for which one
or more elementary symmetric functions of $1, 1/2, \ldots , 1/n$
are integers. In this paper we solve this old problem by showing
that if $n\ge 4$, then none of elementary symmetric functions of
$1, 1/2, \ldots , 1/n$ is an integer.
\end{abstract}



 \section{Introduction}

It is well known that for any integer $n>1$, the harmonic sum
$\sum_{i=1}^n \frac 1i$ is not an integer. In 1946, P. Erd\H os and
I. Niven \cite{ErdosNiven}
 proved that there are only finitely many
positive integers $n$ for which one or more elementary symmetric
functions of $1, 1/2, \ldots , 1/n$ are integers.

In this paper we solve this old problem by showing that
$1/2+1/3+1/6$ is the only case. Let $S(k,n)$ denote the $k-$th
elementary symmetric functions of $1, 1/2, \ldots , 1/n$. Then
$$S(1,1)=1,\quad S(1, 2)=\frac 32,\quad S(2,2)=\frac 12,$$
$$S(1, 3)=\frac{11}6, \quad S(2,3)=1,\quad S(3,3)=\frac 16.$$
For $n\ge 4$ we have the following result:

\begin{theorem}\label{mainthm} Let $k,n$ be positive integers with
$n\ge 4$ and $n\ge k\ge 1$. Then $S(k,n)$ is not an
integer.\end{theorem}

We pose a problem here for further research.

\begin{problem} Find all finite arithmetic progressions $\{
a+mi\}_{i=0}^n$ such that  one or more elementary symmetric
functions of $1/a, 1/(a+2m), \ldots , 1/(a+nm)$ are
integers.\end{problem}

\section{Proofs}

The proof in \cite{ErdosNiven} can give an upper bound of $n$
which is still too large. We improve the proof. First we give
several lemmas.

\begin{lemma}\label{lem1}
$$\sum_{1\le i\le k+1} i=\frac 12 (k+1)(k+2),$$
$$\sum_{1\le i<j\le k+2} ij= \frac 1{24}(k+1)(k+2)(k+3)(3k+8)$$ and
$$\sum_{1\le i<j<s\le k+3} ijs =\frac 1{48}(k+1)(k+2)(k+3)^2(k+4)^2.$$
 \end{lemma}

Proof is omitted. One may deduce these formulas by oneself or by
Mathematica.

\begin{lemma}\label{lem2}(Panaitopol \cite{Panaitopol}) $$\pi(x)<\frac{x}{\log x-1-(\log x)^{-1/2}}\quad \text{ for all } x\geq 6 $$ and
$$\pi(x)>\frac{x}{\log x-1+(\log x)^{-1/2}} \quad \text{ for all }x\geq 59 .$$\end{lemma}

\begin{lemma}\label{lem3} Let $k,n$ be positive integers with
 $n\ge k\ge e\log n+e $. Then $S(k,n)$ is not an
integer.\end{lemma}

\begin{proof} As in P. Erd\H os and
I. Niven \cite{ErdosNiven}, we have
$$S(k,n)\le \frac 1{k!} \left(\sum_{i=1}^n \frac 1i\right)^k.$$
We have
$$\sum_{i=1}^n \frac 1i<1+\sum_{i=2}^n \int_{i-1}^i \frac
1x dx=1+\int_1^n \frac 1x dx=\log n+1$$ and
$$\log k!=\sum_{i=2}^k\log i>\int_1^k \log x dx>k\log k-k\ge k\log (\log n+1),$$
where we use the condition $k\ge e\log n+e $ in the last step. So
$$ \left(\sum_{i=1}^n \frac 1i\right)^k<k!.$$
Hence $S(k,n)<1$. This completes the proof of Lemma
\ref{lem3}.\end{proof}

\begin{lemma}\label{lem4} Let $k,n$ be positive integers with
 $n\ge k>1$. If there exists a prime $p$ such that $$\frac n{k+4}<p\le \frac n {k},\quad p>k+4, p\nmid
3k+8,$$ then $S(k,n)$ is not an integer.\end{lemma}

\begin{proof} Let $p(k+t)\le n<p(k+t+1)$.
It is clear that
$$S(k,n)=\frac 1{p^{k}} S(k, k+t)+\frac{b}{p^{k-1}c},
\quad p\nmid c,$$ where $b,c$ and the following $a, d$ are positive
integers. Since
$$S(k,k)=\frac 1{k!},$$
$$S(k, k+1)=\frac 1{(k+1)! } \sum_{1\le i\le k+1}i = \frac {k+2}{2 (k!)},$$
$$S(k, k+2)=\frac 1{(k+2)!} \sum_{1\le i<j\le k+2}ij = \frac {(k+3)(3k+8)}{24
(k!)}$$ and
$$S(k, k+3)=\frac 1{(k+3)! } \sum_{1\le i<j<s\le k+3} ijs =\frac {(k+3)(k+4)^2}{48 ( k!)},$$
by the conditions we know that $1\le t\le 3$ and $$ S(k,
k+t)=\frac{d}{a}, \quad p\nmid d.$$ Thus
$$S(k,n)=\frac 1{p^{k}} S(k, k+t)+\frac{b}{p^{k -1} c}=\frac{d}{p^k a}+\frac{b}{p^{k-1} c}$$ is
not an integer.

This completes the proof of Lemma \ref{lem4}.\end{proof}

\begin{proof}[Proof of Theorem \ref{mainthm}] It is well known that for any integer $n>1$, the harmonic sum
$\sum_{i=1}^n \frac 1i$ is not an integer. So we may assume that
$k\ge 2$.

By Lemma \ref{lem3} we may assume that $k< e\log n+e$. By Lemma
\ref{lem4} if  there exists a prime $p$ such that
\begin{equation}\label{eqn2}\frac n{k+4}<p\le \frac n k,
\quad p>k+4, p\nmid 3k+8,\end{equation} then Theorem \ref{mainthm}
is true.

First we assume that $n\ge 300000$. Suppose that there exists a
prime $p$ such that
 $$\frac n{k+4}<p\le \frac n k.$$
Now we show that $p>3k+8$. It is sufficient to prove that
$$\frac n{k+4}\ge 3k+8.$$
That is, $n\ge (3k+8)(k+4)$. Since $2\le k\le e\log n+e$, it
suffices to prove that
$$n\ge (3e\log n+3e+8)(e\log n+e+4).$$
Let
$$f(x)=x-(3e\log x+3e+8)(e\log x+e+4).$$
Since
$$ f(300000)>0, \quad f'(x)=1-\frac{6e^2\log x}x-\frac{6e^2+20e}{x}>0$$  for all
$x>300000$, we have $f(x)>0$ for all $x\ge 300000$. Thus it is
enough to prove that
\begin{equation}\label{eqn2}\pi (\frac nk)> \pi (\frac
n{k+4}).\end{equation} Since $n>300000$ and $2\le k\le e\log n+e$,
we have
$$\frac nk>\frac n{k+4}> \frac n{e\log n+e+4}>59.$$
By Lemma \ref{lem2} it suffices to prove that
\begin{eqnarray*}&&\frac{n/k}{\log (n/k) -1+(\log
(n/k))^{-1/2}}\\
&>&\frac{n/(k+4)}{\log (n/(k+4)) -1-(\log
(n/(k+4)))^{-1/2}}.\end{eqnarray*} That is,
\begin{eqnarray*}&&k\left( \log \frac nk -1+(\log
\frac nk)^{-1/2} \right) \\
&<& (k+4) \left( \log \frac n{k+4} -1-( \log \frac n{k+4})^{-1/2}
\right).\end{eqnarray*} That is,
$$k\log \left(1+\frac 4k\right) +4+k \left(\log
\frac nk \right)^{-1/2} +(k+4) \left( \log \frac
n{k+4}\right)^{-1/2}<4\log \frac n{k+4}.$$ Since
$$k\log \left(1+\frac 4k\right)<4, \quad  \log
\frac nk >\log \frac n{k+4},$$ it is enough to prove that
 \begin{equation}\label{eqn3}8+(2k+4) \left( \log \frac n{k+4}\right)^{-1/2}<4 \log \frac
n{k+4}.\end{equation} Let $g(x)=x^{0.3}-e\log x-e-4$. Since
$$g(300000)>0,\quad  xg'(x)=0.3x^{0.3}-e>0$$ for all $x\ge 300000$,
we have $g(x)>0$ for all $x\ge 300000$. Thus
$$\frac n{k+4}\ge \frac n{e\log n+e+4}> n^{0.7},\quad n\ge 300000.$$
By \eqref{eqn3} it is enough to prove
$$8+(2e\log n+2e+4)\left(\log n^{0.7}\right)^{-1/2}\le 4\log n^{0.7},\quad n\ge 300000.$$
That is, \begin{equation}\label{eqn1}8\times 0.7^{1/2} (\log
n)^{1/2} +2e\log n+2e+4\le 4\times 0.7^{3/2} (\log
n)^{3/2}.\end{equation} Let $t=(\log n)^{1/2}$. One may find that
the above inequality is true for $t>3.47603$. So the  inequality
\eqref{eqn1} is true for $n\ge 176802$. Thus \eqref{eqn2} is true
for all $n\ge 300000$.

Now we assume that $n<300000$. By $ k\le e\log n+e$ we have $k\le
37$ and $n\ge e^{k/e-1}$. By Lemma \ref{lem4} we need only to find
a prime $p$ such that
$$kp\le n<(k+4)p, \quad p>k+4, p\nmid 3k+8.$$
In the following $p_i$ denote the $i$-th prime.

{\bf Case 1:} $25\le k\le 37$. Then $139=p_{34}<\frac 1k e^{k/e-1}$.
That is, $kp_{34}<e^{k/e-1}$. If $p$ is a prime with $p\ge 139$,
then $p>k+4$ and $p\nmid 3k+8$. Since $(k+4)p_{1271}\ge
29p_{1271}>300000$, we need only to verify that $kp_{i+1}<(k+4)p_i$
for $34\le i\le 1270$. This can be verified by Mathematica.

{\bf Case 2:} $2\le k\le 24$. We have $(k+4)p_{5134}>300000$. By
Mathematica we can verify that $kp_{i+1}<(k+4)p_i$ for $21\le i\le
5134$. If $73k=kp_{21}\le n<300000$, then there exists a prime $p$
such that $kp\le n<(k+4)p$.  It is easy to verify that for any prime
$p\ge 73$ we have $p>k+4, p\nmid 3k+8$. The remainder cases are
$n\le 73k-1\le 1751$ and $2\le k\le 24$.

It is clear that
$$S(1,1)=1,\quad S(1,n)=S(1,n-1)+\frac 1n,$$ $$ S(n,n)=\frac 1n S(n-1,n-1),$$
$$S(k,n)=S(k,n-1)+\frac 1n S(k-1, n-1),\quad k=2,\dots , n-1.$$
By Mathematica we can verify that $S(k, n)$ is not an integer for
either $4\le n\le 30$ and $2\le k\le n$, or $31\le n\le 1751$ and
$2\le k\le 24$.

This completes the proof of Theorem \ref{mainthm}.\end{proof}

\end{document}